\documentclass[12pt]{amsart}

\usepackage[margin=1.15in]{geometry}

\usepackage{amsmath,amscd,amssymb,amsfonts,latexsym,wasysym, mathrsfs, mathtools,hhline,color}
\usepackage[all, cmtip]{xy}

\usepackage{url}

\definecolor{hot}{RGB}{65,105,225}

\usepackage[pagebackref=true,colorlinks=true, linkcolor=hot ,  citecolor=hot, urlcolor=hot]{hyperref}

\usepackage{graphicx,enumerate}

\newcommand{\CP}{\mathbb{CP}^{n}}

\theoremstyle{plain}
\newtheorem{theorem}{Theorem}[section]
\newtheorem{prop}[theorem]{Proposition}

\newtheorem{lm}[theorem]{Lemma}

\newtheorem{cor}[theorem]{Corollary}

\newtheorem{lemma}[theorem]{Lemma}
\newtheorem{thrm}[theorem]{Theorem}

\theoremstyle{definition}

\newtheorem{defn}[theorem]{Definition}
\newtheorem{que}[theorem]{Question}
\newtheorem{rmk}[theorem]{Remark}

\newtheorem{ex}[theorem]{Example}
\newtheorem*{ex*}{Example}

\def\be{\begin{equation}}
\def\ee{\end{equation}}

\def\bt{\begin{thrm}}
\def\et{\end{thrm}}

\def\bc{\begin{cor}}
\def\ec{\end{cor}}

\def\br{\begin{rmk}}
\def\er{\end{rmk}}

\def\bp{\begin{prop}}
\def\ep{\end{prop}}

\def\bl{\begin{lm}}
\def\el{\end{lm}}

\def\bex{\begin{ex}}
\def\eex{\end{ex}}

\def\bd{\begin{defn}}
\def\ed{\end{defn}}

\newcommand\fp{{\mathfrak p}}

\newcommand\sO{{\mathcal O}}

\newcommand\sP{{\mathcal P}}

\newcommand\sM{{\mathcal M}}

\newcommand\sV{{\mathcal V}}
\newcommand\sL{\mathcal{L}}

\newcommand\sT{\mathcal{T}}

\newcommand{\Fp}{\mathbb{F}_p}
\newcommand\kk{{\mathbb{K}}}

\newcommand\zz{{\mathbb{Z}}}

\newcommand\cc{{\mathbb{C}}}

\DeclareMathOperator{\codim}{codim}              
                  
\DeclareMathOperator{\id}{id}                    

\DeclareMathOperator{\spec}{Spec}

\DeclareMathOperator{\rank}{Rank}

\DeclareMathOperator{\IC}{IC}
\DeclareMathOperator{\Ann}{Ann}
\DeclareMathOperator{\Ass}{Ass}
\DeclareMathOperator{\Perv}{Perv}
\DeclareMathOperator{\Alb}{Alb}
\DeclareMathOperator{\alb}{alb}
\DeclareMathOperator{\rad}{rad}
\DeclareMathOperator{\depth}{depth}

\DeclareMathOperator{\Hom}{Hom}

\DeclareMathOperator{\Char}{Char}

\def\ra{\rightarrow}

\def\bC{\mathbb{C}}
\def\cM{\mathcal{M}}

\def\cH{\mathcal{H}}

\def\lra{\longrightarrow}
\def\bQ{\mathbb{Q}}

\def\bZ{\mathbb{Z}}

\def\bN{\mathbb{N}}

\newcommand{\ubul}{{\,\begin{picture}(-1,1)(-1,-3)\circle*{2}\end{picture}\ }}

\title[Propagation and abelian duality spaces]{Mellin transformation, propagation, \\ and abelian duality spaces}

\author{Yongqiang Liu}
\address{Department of Mathematics, KU Leuven, 
Celestijnenlaan 200B, B-3001 Leuven, Belgium} 
\email{liuyq1117@gmail.com}
\author{Laurentiu Maxim}
\address{Department of Mathematics,
          University of Wisconsin-Madison,
          480 Lincoln Drive, Madison WI 53706-1388, USA.}
\email {maxim@math.wisc.edu}
\author{Botong Wang}
\address{Department of Mathematics,
          University of Wisconsin-Madison,
          480 Lincoln Drive, Madison WI 53706-1388, USA.}
\email {wang@math.wisc.edu}
\date{\today}

\keywords{complex affine torus, perverse sheaf, Mellin transformation, cohomology jump loci, propagation property, generic vanishing, signed Euler characteristic, abelian duality space, abelian variety}

\subjclass[2010]{32S60, 14F17, 55N25, 55U30}

\begin{document}

\maketitle
\begin{abstract}  
For arbitrary field coefficients $\kk$, we show that $\kk$-perverse sheaves on a complex affine torus satisfy the so-called {\it propagation package}, i.e., the generic vanishing property and the signed Euler characteristic property hold, and 
the corresponding cohomology jump loci satisfy the propagation property and codimension lower bound. The main ingredient used in the proof is Gabber-Loeser's Mellin transformation functor for $\kk$-constructible complexes on a complex affine torus, and the fact that it behaves well with respect to perverse sheaves. 

As a concrete topological application of our sheaf-theoretic results, we study homological duality properties of complex algebraic varieties, via {\it abelian duality spaces}. We provide new obstructions on abelian duality spaces by showing that their cohomology jump loci satisfy a propagation package. This is then used to prove that complex abelian varieties are the only complex projective manifolds which are abelian duality spaces. We also construct new examples of abelian duality spaces. For example, we show that if a smooth quasi-projective variety $X$ satisfies a certain Hodge-theoretic condition and it admits a proper semi-small map (e.g., a closed embedding or a finite map) to a complex affine torus, then $X$ is an abelian duality space.
\end{abstract}
 
 \tableofcontents


\section{Introduction}\label{intro}

Cohomology jump loci of rank-one local systems provide a unifying framework for the study of a host of questions concerning homotopy types of complex algebraic varieties. In particular, they can be used to tackle Serre's problem concerning groups which can be realized as fundamental groups of complex quasi-projective manifolds.

By the classical {Albanese map} 
construction (e.g., see \cite{Iit}), cohomology jump loci of a complex quasi-projective manifold can be understood via cohomology jump loci of {constructible complexes of sheaves} (or, 
if the Albanese map is proper, of {\it perverse sheaves}) on a semi-abelian variety. 
This motivates our investigation of cohomology jump loci of such complexes, and in particular of cohomology jump loci of perverse sheaves. Throughout this paper, we restrict ourselves to the study of perverse sheaves on a complex affine torus, but we allow any field of coefficients. Specifically, we show that for any field $\kk$, the cohomology jump loci of $\kk$-perverse sheaves on a complex affine torus satisfy the so-called {\it propagation package}, i.e., a list of properties which provide new obstructions on the category of perverse sheaves on a complex affine torus. The present work parallels results of Gabber-Loeser in  the $\ell$-adic setting, see \cite{GL}. 

As a concrete topological application of our sheaf-theoretic results, we study homological duality properties of complex algebraic varieties, via {\it abelian duality spaces}. Abelian duality spaces have been recently introduced in \cite{DSY}, by analogy with the duality spaces of Bieri-Eckmann \cite{BE}, and they are particularly useful for explaining (non-)vanishing properties of the cohomology of rank-one local systems. We pay a particular attention to the corresponding {\it realization problem}, which is aimed at investigating which smooth complex algebraic varieties are abelian duality spaces. We obtain here new obstructions on abelian duality spaces via their cohomology jump loci. Specifically, we show that cohomology jump loci of abelian duality spaces satisfy a propagation package. This is then used to give a complete answer to the realization problem in the smooth projective (or compact K\"ahler) context, while at the same time providing a new topological description of complex abelian varieties. We also identify a general class of complex algebraic varieties which are abelian duality spaces; this includes the classical examples of complements of complex essential hyperplane arrangements and complements of complex toric arrangements. 


In the remaining of the introduction, we give a more detailed account of the results contained in this paper.

\subsection{Perverse sheaves on complex affine torus}
Let $T= (\bC^*)^N$ be  an  $N$-dimensional complex affine torus. Fix a field of coefficients $\kk$, and set $$\Gamma_T:= \kk[\pi_1(T)].$$ 
We let $\spec \Gamma_T$ denote the {\it maximal} spectrum of $\Gamma_T$, and note that 
$\spec \Gamma_T$ can be identified with the {\it character variety} $\Char_{\kk}:=\Hom(\pi_1,\kk^*)$ when $\kk$ is algebraically closed. For future reference, we let $\Char:=\Char_{\bC}$.
 To any point $\chi \in \spec \Gamma_T$ with residue field $\kk_{\chi} =\Gamma_T/ \chi $, one associates a rank-one local system $L_{\chi}$ of $\kk_{\chi}$-vector spaces on $T$. 
 
\bd \label{perversejump} For any perverse sheaf $\sP\in \Perv(T,\kk)$, the {\it cohomology jump loci of $\sP$} are defined as
 $$\sV^i(T,\sP): = \{\chi \in \spec \Gamma_T \mid H^i(T, \sP \otimes_\kk L_\chi) \neq 0\}.$$
 \ed 
One of the goals of this paper is to show that these jump loci satisfy a list of properties, which we refer to as the {\it propagation package}. In particular, such properties impose new obstructions on the category of perverse sheaves on the  complex affine torus $T$. Specifically, 
 the following result holds:
 \bt  \label{main} Let $\kk$ be a field.  For any perverse sheaf $\sP\in \Perv(T,\kk)$, the cohomology jump loci of $\sP$ satisfy the following properties:
\begin{enumerate}
\item[(i)] {\it Propagation property}:
$$ 
\spec \Gamma_T \supseteq \sV^{0}(T,\sP) \supseteq \sV^{-1}(T,\sP) \supseteq \cdots \supseteq \sV^{-N}(T,\sP).
$$
\item[(ii)] {\it Codimension lower bound}: for any $i\geq 0$, $$ \codim \sV^{-i}(T,\sP) \geq i.$$
\item[(iii)]  If  $V$ is an irreducible component of $\sV^0(T,\sP)$ of codimension $d$, then $V \subset \sV^{-d}(T,\sP).$ In particular, $V$ is also an irreducible component of $\sV^{-d}(T,\sP)$.
\item[(iv)]  If $\sV^0(T,\sP)$ has codimension $d\geq 0$, then $$\sV^0(T,\sP)= \sV^{-1}(T,\sP)= \cdots = \sV^{-d}(T,\sP) \neq \sV^{-d-1}(T,\sP).$$  If moreover $\codim \sV^{-k}(T,\sP)>k$ for all $k>d$, then $\sV^0(T,\sP)$ is pure-dimensional.
\item[(v)] {\it Generic vanishing}: there exists a non-empty Zariski open subset $U \subset \spec \Gamma_T $ such that,  for any maximal ideal 
$\chi\in U$,  $H^{i}(T, \sP\otimes_{\kk} L_\chi)=0$ for all $i\neq 0$.
\item[(vi)] {\it Signed Euler characteristic property}:  $$\chi(T,\sP)\geq 0.$$
Moreover, the equality holds if and only if $\sV^0(T,\sP) \neq \spec \Gamma_T$.
\end{enumerate} 
\et

The above result can be viewed as a topological counterpart of similar properties satisfied by the Green-Lazarsfeld {\it algebraic jump loci} of topologically trivial line bundles, see \cite{GLa,GrLa}. The propagation package for algebraic jump loci was used by Chen-Hacon \cite{CH} to give a birational characterization of abelian varieties. Various topological applications of the above theorem will be given in Section \ref{apl}.

The main ingredient used for proving Theorem \ref{main} is Gabber-Loeser's Mellin transformation functor for $\kk$-constructible complexes on the complex affine torus $T$,  and the fact that it behaves well with respect to perverse sheaves. More precisely, in Theorem \ref{GL} we extend the $\ell$-adic context of Gabber-Loeser \cite[Theorem 3.4.1]{GL} to arbitrary field coefficients $\kk$, and show that the Mellin transformation is a $t$-exact functor for $\kk$-perverse sheaves on a complex affine torus. This fact is then combined with commutative algebra techniques (see Proposition \ref{prop}) to obtain a proof of Theorem \ref{main}.
 
 \br  The signed Euler characteristic property was proved in the $\ell$-adic context by Gabber-Loeser, see \cite[Corollary 3.4.4]{GL} (though it is credited in loc.cit. to Laumon). For $\bC$-perverse sheaves on semi-abelian varieties, a similar signed Euler characteristic property was obtained in  \cite[Corollary 1.4]{FK}.  The generic vanishing property for perverse sheaves on affine tori, abelian varieties and, respectively, semi-abelian varieties was proved in various settings (often in connection with the Tannakian formalism), see, e.g.,  
 \cite[Theorem 2.1]{Kra}, \cite[Theorem 1.1]{KW}, \cite[Corollary 7.5]{Sch}, \cite[Vanishing Theorem]{We}, \cite[Theorem 1.1]{BSS}, \cite{LMW}.
 The codimension lower bound is implicit in \cite[Proposition 6.3.2]{GL} in the $\ell$-adic setting, on affine tori defined over an algebraically closed field of positive characteristic. 
If $\kk$ is a field of characteristic zero, an  improved codimension lower bound for $\kk$-perverse sheaves on a complex abelian variety was obtained by Schnell \cite{Sch}, see also \cite[Theorem 1.3]{BSS}, and it was used to completely characterize perverse sheaves in the abelian context. 
 \er
 
  \br
When $\kk=\cc$, it follows from \cite[Theorem 10.1.1]{BW17} that for any $i$, $\sV^i(T,\sP)$ is a finite union of translated subtori. Moreover, the above statement still holds when $T$ is any smooth complex variety (in which case, $\Gamma_T=\cc[H_1(T, \zz)]$) and $\sP$ is any bounded constructible complex over $\cc$ on $T$. 
 \er

 
\subsection{Abelian duality spaces} Another important application of the $t$-exactness of the Mellin transformation is to detect {\it abelian duality spaces}. 

Abelian duality spaces were introduced by Denham-Suciu-Yuzvinsky in \cite{DSY}, by analogy with the {\it duality spaces} of Bieri-Eckmann \cite{BE}, and they were used in loc.cit. for explaining some previously conjectural behavior of the cohomology of abelian representations. Specifically, it was shown in \cite[Theorem 1.1]{DSY} that abelian duality spaces satisfy a propagation property for their cohomology jump loci, similar to that of Theorem \ref{main}($i$). In this paper we show that abelian duality spaces satisfy the whole propagation package.

  
Let $X$ be a connected finite CW complex. Denote $\pi_1(X)$ by $G$, and its abelianization by $G^{ab}$. There is a canonical $\bZ [G^{ab}]$-local coefficient system on $X$, whose monodromy action is given by the composition of the quotient $G\to G^{ab}$ with the natural multiplication  $G^{ab}\times \bZ[ G^{ab}] \to \bZ[ G^{ab}]$. 
\bd \label{DSY} \cite{DSY} A finite connected CW complex $X$ is called an {\it abelian duality space} of dimension $n$ if: 
\begin{itemize} \item[(a)] $H^i(X, \bZ [G^{ab}])=0$ for all $i\neq n$,  \item[(b)] $H^n(X, \bZ[ G^{ab}])$ is non-zero and a torsion-free $\bZ$-module.\end{itemize} \ed

 This is equivalent to the statement that the (integral) compactly supported cohomology of the universal abelian cover of $X$ is concentrated in a single dimension, $n$,  where it is torsion-free.

A {\it duality space} \cite{BE} is defined similarly, by using $G$ instead of $G^{ab}$ in the above definition (hence by replacing the universal abelian cover by the universal cover). 
 As pointed out in \cite{DSY}, there are duality spaces which are not abelian duality spaces, and the other way around. In Section \ref{abd}, we extend the notion of abelian duality space, and introduce what we call {\it partially abelian duality spaces}  with respect to a homomorphism $\phi:G \to G'$ to an abelian group $G'$. 

Duality spaces and abelian duality spaces enjoy (homological) duality properties similar to Serre duality for projective varieties (and less restrictive than Poincar\'e duality). 
For example, if $X$ is an abelian duality space of dimension $n$,
then for any $\bZ[G^{ab}]$-module $A$, one has isomorphisms (\cite{DSY}):
$$H^i(X,A) \cong H_{n-i}(G^{ab},B \otimes_{\bZ} A),$$
where $B:=H^n(X, \bZ [G^{ab}])$ is called the {\it dualizing $ \bZ[G^{ab}]$-module}. 

\vspace{.55mm}

Our motivation for studying abelian duality spaces comes from the following {\it realization question}: which smooth complex algebraic varieties are (abelian) duality spaces? Partial answers in the projective/K\"ahler context are discussed in Section \ref{real} below. As an extension of results from \cite{DSY}, we also provide here further obstructions on abelian duality spaces via their cohomology jump loci, see Theorem \ref{ad}. 

\bd  Let $\kk$ be either $\bZ$ or a field. Set $\Gamma_{G^{ab}}= \kk[G^{ab}],$ with maximal spectrum  $\spec \Gamma_{G^{ab}}$. The {\it cohomology jump loci of $X$} with $\kk$-coefficients are given by
 $$\sV^i(X): = \{\chi \in \spec \Gamma_{G^{ab}} \mid H^i(X,  L_\chi) \neq 0\},$$
  where for any  maximal ideal $\chi \in \spec \Gamma_{G^{ab}}$ with residue field $\kk_{\chi}  =\Gamma_{G^{ab}}/ \chi $, $L_\chi$ denotes the corresponding rank $1$ local system of $\kk_{\chi}$-vector spaces on $X$. 
 \ed 
 
In Theorem \ref{pad} we show that abelian duality spaces satisfy the propagation package, namely we have the following result.
\bt \label{ad} Let $X$ be an abelian duality space of dimension $n$. Fix $\kk$ as either $\bZ$ or any field. Then the cohomology jump loci of $X$ with $\kk$-coefficients satisfy the following  properties:
\begin{enumerate}
\item[(i)] {\it Propagation property}:
$$ 
\spec \Gamma_{G^{ab}} \supseteq \sV^{n}(X) \supseteq \sV^{n-1}(X) \supseteq \cdots \supseteq \sV^{0}(X).
$$
\item[(ii)] {\it Codimension lower bound}: for any $i\geq 0$, $$ \codim \sV^{n-i}(X)= b_1(X)-\dim  \sV^{n-i}(X) \geq i.$$
\item[(iii)]  If  $V$ is an irreducible component of $\sV^n(X)$ of codimension $d$, then $V \subset \sV^{n-d}(X).$ In particular, $V$ is also an irreducible component of $\sV^{n-d}(X)$.
\item[(iv)]  If $\sV^n(X)$ has codimension $d\geq 0$, then $$\sV^n(X)= \sV^{n-1}(X)= \cdots = \sV^{n-d}(X) \neq \sV^{n-d-1}(X).$$  If moreover $\codim \sV^{n-k}(X)>k$ for all $k>d$, then $\sV^n(X)$ is pure-dimensional. 
\item[(v)] {\it Generic vanishing}: there exists a non-empty Zariski open subset $U \subset \spec \Gamma_{G^{ab}}$ such that,  for any maximal ideal $\chi\in U$,  $H^{i}(X,  L_\chi)=0$ for all $i\neq n$.
\item[(vi)] {\it Signed Euler characteristic property}:  $$  (-1)^n\chi(X)\geq 0.$$
Moreover, the equality holds if and only if $\sV^n(X) \neq \spec \Gamma_{G^{ab}}$.
\end{enumerate} 
 \et

As already mentioned above, it was also shown in \cite[Theorem 1.1]{DSY} that abelian duality spaces satisfy the propagation property for their jumping loci. This propagation property already imposes strong restrictions on the topology of $X$. For example, it has the following implications on the Betti numbers and Euler characteristics. 
\begin{theorem}[\cite{DSY}]\label{Bettinumbers}
If $X$ is an abelian duality space of dimension $n$, then:
\begin{itemize}
\item[(1)] $b_i(X) > 0$, for $0 \leq i \leq n$,
\item[(2)] $b_1(X) \geq  n$,
\item[(3)]  $(-1)^n \chi(X) \geq 0$. 
\end{itemize}
\end{theorem}

In the recent preprint \cite{DS} by Denham-Suciu,  it is shown that complements of certain classes of arrangements of hypersurfaces (including projective hyperplane arrangements, toric arrangements and elliptic arrangements) are both duality and abelian duality spaces. In particular, the whole propagation package of Theorem \ref{ad} is valid for these examples.

\medskip

The following result gives a criterion to detect abelian duality spaces from a relative perspective (see Theorem \ref{veryaffine}). 
 \begin{theorem} \label{abel}
Let $X$ be a smooth complex quasi-projective variety of dimension $n$. Assume that  the mixed Hodge structure on $H^1(X, \bQ)$ is pure of type $(1,1)$, or equivalently, there exists a smooth compactification $\overline{X}$ of $X$ with $b_1(\overline{X})=0$.   If $X$ admits a proper semi-small map $f: X\to T=(\bC^*)^N$ (e.g., a finite map or a closed embedding), then $X$ is an abelian duality space of dimension $n$. 
\end{theorem}
In particular, Theorem \ref{abel} shows that the complements of hyperplane arrangements and of toric arrangements are abelian duality spaces (but our result does not apply to elliptic arrangement complements).


\subsection{The realization problem for abelian duality spaces}\label{real}
There are plenty of projective manifolds that are duality spaces. For example, in dimension one, all genus $g\geq 1$ smooth complex projective curves are duality spaces. On the other hand, only genus one curves are abelian duality spaces. More generally, any projective manifold which is a ball quotient is a duality space, e.g. see \cite{Yau} for such examples.  On the other hand, we show that complex projective manifolds which are abelian duality spaces are rather rare. More precisely, as a direct application of the propagation property for abelian duality spaces, we have the following result (see Theorem \ref{av}). 
\begin{theorem}\label{abp}
Let $X$ be a complex projective manifold. Then $X$ is an abelian duality space if and only if $X$ is a complex abelian variety. 
\end{theorem}
Theorem \ref{abp} can be viewed as a topological counterpart of the Chen-Hacon {\it birational} characterization of abelian varieties \cite{CH}. Furthermore, the above result can be extended to the K\"ahler context, namely we show in Theorem \ref{av} that compact complex tori are the only compact K\"ahler manifolds that are abelian duality spaces. 

\subsection{Summary}

The paper is organized as follows. 

In Section \ref{alg}, we recall the definition of several algebraic notions (including fitting ideals, jumping ideals and, resp., cohomology jump loci of a bounded complex of $R$-modules with finitely generated cohomology) and prove an algebraic version of the propagation package.  

In Section \ref{top}, we recall the definition of the Mellin transformation functor for $\kk$-constructible complexes on a complex affine torus, and show that it is $t$-exact (see Theorem \ref{GL}). As an application of this fact, we prove Theorem \ref{main} and discuss several topological implications.

In Section \ref{abd}, we introduce the notion of partially abelian duality space as a generalization of abelian duality spaces, and show in Theorem \ref{pad} that such spaces satisfy the propagation package. Furthermore, we use once more the Mellin transformation functor to construct new examples of (partially) abelian duality spaces in Theorem \ref{veryaffine}. 

Section \ref{sex} is devoted to a discussion of several concrete examples of (partially) abelian duality spaces.

In Section \ref{projective}, we prove that compact complex tori are the only compact K\"ahler manifolds that are abelian duality spaces. In particular, abelian varieties are the only complex projective manifolds which are abelian duality spaces. We also point out that the Singer conjecture admits an equivalent formulation in terms of duality spaces.

\medskip

\textbf{Acknowledgments.} We are grateful to Nero Budur, J\"org Sch\"urmann and Alex Suciu for useful discussions. The authors  thank the Mathematics Department at the University of Wisconsin-Madison, East China Normal University (Shanghai, China) and University of Science and Technology of China (Hefei, China)  for hospitality during the preparation of this work. 
 

\section{Cohomology jump loci of an $R$-module}\label{alg}
Let $R$ be a noetherian domain, and let $E^\ubul$ be a bounded above complex of $R$-modules with finitely generated cohomology $R$-modules. In this section, we recall the notion of cohomology jump loci for the complex $E^\ubul$. 

By a result of Mumford (see \cite[III.12.3]{Ha}), there exists a bounded above complex $F^\ubul$ of finitely generated free $R$-modules, which is a quasi-isomorphic to $E^\ubul$. 

\bd For any integer $k$ and a map $\phi$ of free $R$-modules, let $I_{k}\phi$ denote the determinantal ideal of $\phi$ (i.e., the ideal of minors of size $k$ of the matrix of $\phi$), see \cite[p.492-493]{E}.  Then, in the above notations, the {\it degree $i$ fitting ideal} of $E^\ubul$ is defined as:
$$I^i(E^\ubul)= I_{\rank \partial^{i}}( \partial^i),$$
and the {\it  degree $i$ jumping ideal} of $E^\ubul$ is defined as: $$J^{i} (E^\ubul)=I_{\rank(F^{i})}(\partial^{i-1}\oplus \partial^{i}),
$$
where  $\partial^{i-1}: F^{i-1}\to F^{i}$ and $\partial^{i}: F^{i}\to F^{i+1}$ are differentials of the complex $F^\ubul$. 

We define the {\it $i$-th cohomology jump loci} of $E^\ubul$ as 
$$  \sV^{i} (E^\ubul) :=  \spec(R/  \rad J^i(E^\ubul)),  
$$ 
where  $\rad I$ denotes the radical ideal of the ideal $I$ in $R$, and $\spec$ denotes the maximal spectrum. By definition, $\sV^i(E^\ubul)$ is naturally a subset of $\spec(R)$ induced by the quotient map $R\to R/ \rad J^i(E^\ubul)$. 
\ed 
It is known that $I^i(E^\ubul)$,  $J^i(E^\ubul)$ and $\sV^i(E^\ubul)$ do not depend on the choice of the finitely generated free resolution $F^\ubul$ for $E^\ubul$, see \cite[Section 20.2]{E} and \cite[Section 2]{BW}.

\br\label{alt}
An equivalent definition of $\sV^{i} (E^\ubul)$ can be given as follows:
$$ \sV^{i} (E^\ubul):= \{ \chi \in \spec R \mid H^i (F^\ubul \otimes _R R/\chi) \neq 0 \},$$ 
 with $F^\ubul $ as above; see \cite[Corollary 2.5]{BW}.\er
 

Let $M$ be a finitely generated $R$-module. Then $M$ can be viewed as a complex of $R$-modules concentrated in degree zero. There exists a finitely generated free resolution for $M$: 
$$ \cdots \to F^{-i-1} \overset{\partial^{-i-1}}{\to} F^{-i} \overset{\partial^{-i}}{\to} F^{-i+1} \to  \cdots \to F^{-1} \overset{\partial^{-1}}{\to} F^0  \to  0 .$$
Let $\Ann M$  denote the annihilator ideal of $M$ and let $\Ass(M)$ denote the collection of prime ideals of $R$ associated to $M$, i.e.,  if $\fp$ is a prime ideal in $R$ then $\fp \in \Ass(M)$ if $\fp$ is the annihilator of an element of $M$.
Let us now recall the following results from \cite[Theorem 20.9, Corollary 20.12, Corollary 20.14]{E}.
\bl \label{E}    Let
$M$ be a finitely generated $R$-module.  Assume that  there exists a finitely generated free resolution of $M$ of finite length. Then the following properties hold:
\begin{enumerate}
\item[(i)]   
$$\rad I^{-i}(M) \subset\rad I^{-i-1}(M)  \text{ for any } i\geq 1.$$
\item[(ii)]  
 $$\depth(I^{-i}(M)) \geq i \  \text{ for any  } i\geq 1.$$
\item[(iii)]  $$\rank F^{-i}= \rank \partial^{-i} +\rank \partial^{-i-1}  \text{ for any } i\geq 1.$$
\item[(iv)]    Let $\fp$ be a prime ideal of $R$, and let $\fp_\fp$ denote the corresponding maximal ideal in $R_\fp$.  If  $d= \depth \fp_\fp$, then $\fp \in \Ass(M)$ if and only if $I^{-d}(M) \subset \fp. $ 
\item[(v)]  Set $d=\depth \Ann(M)$. Then  $\depth \fp_\fp =d$ for all $\fp\in \Ass(M)$ if and only if $\depth I^{-k}(M) >k$ for all $k>d$.
\item[(vi)]  
If $M$ has a non-zero annihilator, then $\partial^{-1}$ has rank equal to the rank of $F^0$. In this case, writing $d=\depth  I_{\rank F^0}(\partial^{-1})$, we have  
$$ \rad I^{-1}(M)= \cdots = \rad I^{-d}(M) \neq  \rad   I^{-d-1}(M).$$
\end{enumerate}
\el

\br  Let $\kk$ be a principal ideal domain (PID).  In this paper, we always work with Laurent polynomial rings $R=\kk[t_{1},t_1^{-1}, \cdots, t_N,t_N^{-1}]$, for some positive integer $N$.  Note that any finitely generated projective $R$-module  is free in this case \cite[Theorem 8.13]{BG}, hence every finitely generated $R$-module admits a free resolution with length $\leq N$ for $\kk$ a field and length $\leq N+1$ for $\kk$ a PID. Moreover,  since $R$ is a Cohen-Macaulay ring in this case, the depth of  any ideal in $R$ coincides with its codimension.  
\er

Next we translate the above algebraic statement for fitting ideals into a corresponding statement for jump loci.
\bp \label{prop}  Assume that $R$ is a Cohen-Macaulay domain. Then, with the same assumptions and notations from Lemma \ref{E}, we have that $J^0(M) \subset I^{-1}(M)$, $\rad J^0(M)=  \rad \Ann(M)$, and for any $i\geq 1$, $$\rad I^{-i}(M)= \rad J^{-i}(M).$$ Moreover, the following properties hold:
\begin{enumerate}
\item[(i)]    {\it Propagation property}:
$$ 
\spec(R) \supseteq \sV^{0}(M) \supseteq \sV^{-1}(M) \supseteq \cdots \supseteq \sV^{-i}(M ) \supseteq \sV^{-i-1}(M) \supseteq \cdots.
$$
\item[(ii)] {\it Codimension lower bound}: for any $i\geq 0$, $$ \codim \sV^{-i}(M) \geq i.$$
\item[(iii)]  If  $V$ is an irreducible component of $\sV^0(M)$ of codimension $d$, then $V \subset \sV^{-d}(M).$ In particular, $V$ is also an irreducible component of $\sV^{-d}(M)$. Here the irreducible component should be understood as the minimal prime ideal containing $\rad J^0(M)$ and $\rad J^{-d}(M)$, respectively.
\item[(iv)]  If $\sV^0(M)$ has codimension $d\geq 0$, then $$\sV^0(M)= \sV^{-1}(M)= \cdots = \sV^{-d}(M) \neq \sV^{-d-1}(M).$$  If moreover $\codim \sV^{-k}(M)>k$ for all $k>d$, then $\sV^0(M)$ is pure-dimensional.

\item[(v)] {\it Generic vanishing}: there exists a non-empty Zariski open subset $U \subset \spec(R)$ such that,  for any maximal ideal $\chi \in U$,  $H^{i}(F^\ubul \otimes_R R/\chi)=0$ for all $i\neq 0$.
\item[(vi)] {\it Signed Euler characteristic property}: $$  \sum_{i\leq 0} (-1)^i \rank F^i\geq 0.$$
\end{enumerate}
\ep 

\begin{proof}
By Lemma \ref{E}$(iii)$, we have that $\rank F^{-i}= \rank \partial^{-i} +\rank \partial^{-i-1}.$ It follows that, for any $i\geq 1$,  $$ J^{-i}(M)= I^{-i}(M) \cdot I^{-i-1}(M),$$
so $$ \rad J^{-i}(M)=(\rad  I^{-i}(M) ) \cap (\rad  I^{-i-1}(M)) = \rad  I^{-i}(M), $$
where the last equality follows from  Lemma \ref{E}$(i)$.  For $i=0$,   $$J^0(M)=I_{\rank F^0}(\partial^{-1}) \subset I_{\rank \partial^{-1}}(\partial^{-1}) =I^{-1}(M).$$
The identification  $\rad J^0(M)=  \rad \Ann(M)$ follows from \cite[Proposition 20.7]{E}.
The inclusion $\sV^0(M) \supseteq \sV^{-1}(M)$ follows from  the fact that $J^0(M) \subset I^{-1}(M) $ together with $\rad I^{-1}(M)=\rad J^{-1}(M).$

Let us now recall the following important relation between $\Ass (M)$ and $\rad \Ann M$,  see \cite[Theorem 6.5]{Ma}:
\begin{itemize}
\item[(a)] If a prime ideal $\fp \in \Ass(M)$, then   $\rad \Ann M  \subset \fp .$ 
\item[(b)] The set of minimal prime ideals, which contain $\rad \Ann M$, and the set of minimal elements of $\Ass(M)$   coincide.
\end{itemize}  
Then $(iv)$ follows from Lemma \ref{E}$(v)$ and Lemma \ref{E}$(vi)$.
  If  $V$ is an irreducible component of $\sV^0(M)$ of codimension $d$, then Lemma \ref{E}$(iv)$ yields that $V \subset \sV^{-d}(M)$. Hence $V$ is also an irreducible component of $\sV^{-d}(M)$ because of the codimension bound.

For $(v)$,  one can take $U=\spec (R) \setminus  \sV^{-1}(M)$, after noting that $\spec R \neq \sV^{-1}(M)$ because of the codimension bound.  Then $(vi)$ follows from $(v)$.

The remaining properties follow directly from Lemma \ref{E} and the definition of jump loci. 
\end{proof}


\section{Perverse sheaves on a complex affine torus}\label{top}
For any complex algebraic variety $X$ and any commutative Noetherian ring $R$, we denote by  $D^{b}_{c}(X,R)$ the derived category of bounded cohomologically $R$-constructible complexes of sheaves on $X$. 

In this section, we fix a field of coefficients $\kk$, e.g.,  $\bC$ or $\Fp=\bZ/p\bZ$ (for a prime $p$).


\subsection{Mellin transformation functor and the propagation property}\label{MT}

Let $T= (\bC^{\ast})^{N}$ be an $N$-dimensional complex affine torus. Set $$\Gamma_T:=\kk[\pi_1(T)]\cong \kk[t_{1},t_1^{-1}, \cdots, t_N,t_N^{-1}].$$ 
Let $\sL_T$ be the rank $1$ local system of $\Gamma_T$-modules on $T$ associated to the tautological character $\tau:\pi_1(T) \to \Gamma_T^*$, which maps the generators of $\pi_1(T)$ to the multiplication by the corresponding variables of the Laurent polynomial ring $\Gamma_T$.

 \bd \cite{GL} The {\it Mellin transformation functor} $\sM_{\ast}: D^{b}_{c}(T, \kk) \to D^{b}_{coh}(\Gamma_T)$ is defined as
$$\sM_\ast(\sP) := Ra_\ast( \sL_T\otimes_{\kk}\sP), $$
where  $D^{b}_{coh}(\Gamma_T)$ denotes  the bounded coherent complexes of $\Gamma_T$-modules, and $a :T\to pt$ is  the constant map. \ed

Let  $\Perv(T,\kk)$ denote the category of perverse sheaves with $\kk$-coefficients on $T$. The following result, which is adapted from the $\ell$-adic context of Gabber-Loeser \cite[Theorem 3.4.1]{GL}, shows that the Mellin transformation functor behaves well with respect to the perverse $t$-structure on $D^{b}_{c}(T, \kk)$ and the standard $t$-structure on $D^{b}_{coh}(\Gamma_T)$, namely:
\bt \label{GL} Let $\kk$ be a fixed field. The Mellin transformation functor $\sM_\ast$ is $t$-exact, i.e., for any $\sP\in \Perv(T,\kk)$, we have that $H^{i}(\sM_{\ast}(\sP))=0$ for $i\neq 0$.
\et
 
\begin{proof}
First, we reduce the proof to the case when $\kk$ is algebraically closed. Obviously, the Mellin transformation commutes with field extension, i.e., $$\sM_*(\sP\otimes_\kk \overline{\kk})\cong \sM_*(\sP)\otimes_\kk \overline{\kk}.$$ Once we know that $\sM_*(\sP\otimes_\kk \overline{\kk})$ is concentrated in degree zero, it follows that the same holds for $\sM_*(\sP)$. 

Secondly, instead of working with $\Gamma_T$ and local system $\sL_T$, we localize it at the maximal ideal $m_\lambda=(t_1-\lambda_1, \ldots, t_N-\lambda_N)$. Denote the localization $(\Gamma_T)_{m_\lambda}$ by $\Gamma^\lambda_T$, and let 
$$\sL_T^\lambda=\sL_T\otimes_{\Gamma_T}\Gamma_T^\lambda.$$
Denote the corresponding Mellin transformation functor by $\sM^\lambda_*$, that is,
$$\sM^\lambda_*(\sP) := Ra_\ast( \sL^\lambda_T\otimes_{\kk}\sP)\in D^b_{coh}(\Gamma_T^\lambda).$$
Since $\Gamma^\lambda_T$ is a flat $\Gamma_T$-module, we have that 
$$\sM^\lambda_*(\sP)\cong \sM_*(\sP)\stackrel{L}{\otimes}_{\Gamma_T} \Gamma_T^\lambda\in D^b_{coh}(\Gamma_T^\lambda).$$
Since a $\Gamma_T$-module is zero if and only if its localization at every maximal ideal $m_\lambda$ is equal to zero, it suffices to show that $\sM^\lambda_*$ is $t$-exact for every $\lambda$. This fact is proved in \cite[Section 3.4]{GL} in the $\ell$-adic setting. For the convenience of the reader, we sketch here the arguments adapted to our context. 

The right $t$-exactness of $\sM^\lambda_*$ follows from the Artin vanishing theorem. 

For the left $t$-exactness, we first reduce the proof to the case of a simple perverse sheaf. The abelian category of $\kk$-perverse sheaves is Artinian and Noetherian, and hence we have a well-defined notion of the length of $\kk$-perverse sheaves. Assume that $\sP$ is not simple and let $\sP^\prime$ be  a non-trivial sub-perverse sheaf of $\sP$. Then we have a short exact sequence of $\kk$-perverse sheaves:
$$  0 \to \sP^\prime \to \sP \to \sP^{\prime \prime} \to 0.$$
Applying the Mellin transformation functor to the corresponding distinguished triangle, we get a long exact sequence of $\Gamma_T^\lambda$-modules:
$$ \cdots  \to H^i (\sM^\lambda_*(\sP^\prime)) \to H^i (\sM^\lambda_*(\sP)) \to H^i (\sM^\lambda_*(\sP^{\prime \prime})) \to H^{i+1} (\sM^\lambda_*(\sP^\prime)) \to \cdots$$
By the exactness
of the long exact sequence, it is clear that if the claim holds for  $\sP^\prime$ and $\sP^{\prime\prime}$, then it also holds for $\sP$.  By induction on the length of $\kk$-perverse sheaves, the proof  of left $t$-exactness is reduced to the case of simple perverse sheaves.

Let  $\sP$ be a simple $\kk$-perverse sheaf on $T$.  We need to show the vanishing $H^i(\sM^\lambda_*(\sP))=0$ for all $i<0$. This claim is proved by using induction on the dimension of the torus $T$ as follows. 

Let $p: T\to T'$ be a projection onto a $(N-1)$-dimensional torus $T'$, given by the first $N-1$ coordinates. Since the relative dimension of the affine morphism $p$ is one, the only possibly non-trivial perverse cohomology sheaves $^p \cH^i Rp_*(\sP)$ may appear in the range $i\in \{-1, 0\}$. Since $p$ is a smooth morphism of relative dimension one, it follows from \cite[Page 111]{BBD} that there is a canonical monomorphism of perverse sheaves
\be\label{mn} p^*\left( ^p \cH^{-1} Rp_*(\sP)\right)[1]\hookrightarrow \sP.\ee
So if $^p \cH^{-1} Rp_*(\sP)$ is non-zero, then since $\sP$ is simple, the monomorphism (\ref{mn}) is an isomorphism:
$$p^*\left( ^p \cH^{-1} Rp_*(\sP)\right)[1]\cong \sP.$$
The desired vanishing in this case follows from the corresponding version of \cite[Proposition 3.1.3(d)]{GL}, by using the induction hypothesis applied to the perverse sheaf $^p \cH^{-1} Rp_*(\sP)$ on $T'$. 

On the other hand, if $^p \cH^{-1} Rp_*(\sP)$ is zero, then $Rp_*(\sP)$ is a perverse sheaf on $T'$. 
Notice that we have the isomorphism (compare with \cite[Proposition 3.1.3(c)]{GL})
$$\sM^{\lambda'}_*(Rp_*(\sP))\cong \sM^\lambda_*(\sP)\stackrel{L}{\otimes}_{\Gamma^\lambda_{T}} \Gamma^{\lambda'}_{T'},$$ where $ \Gamma^{\lambda'}_{T'}$ is the localization of $ \Gamma_{T'}$ at the corresponding maximal ideal $m_{\lambda'}$ obtained from $m_{\lambda}$ by forgetting the missing (last) variable.
Let $f\in \Gamma^\lambda_T$ be such that $\Gamma^{\lambda'}_{T'}=\Gamma^\lambda_T/(f)$. Then the complex $\Gamma^\lambda_T\stackrel{f}{\longrightarrow}\Gamma^\lambda_T$ is a free resolution of $\Gamma^{\lambda'}_{T'}$. Thus, 
\begin{equation}\label{tensor}
\sM^{\lambda'}_*(Rp_*(\sP))\cong \sM^\lambda_*(\sP)\otimes_{\Gamma^\lambda_T}(\Gamma^\lambda_T\stackrel{f}{\longrightarrow}\Gamma^\lambda_T).
\end{equation}
At this point, since  $Rp_*(\sP)$ is a perverse sheaf on $T'$, we can apply the induction hypothesis for $Rp_*(\sP)$ to obtain that $H^i(\sM^{\lambda'}_*(Rp_*(\sP)))=0$ for $i<0$. Hence by (\ref{tensor}), we get that the multiplication by $f$
$$H^i(\sM^\lambda_*(\sP))\stackrel{f}{\longrightarrow}H^i(\sM^\lambda_*(\sP))$$
is surjective for $i<0$. Therefore, by Nakayama Lemma for the local ring $\Gamma^\lambda_T$, we obtain the desired vanishing. 
\end{proof}

We can now prove Theorem \ref{main} from Introduction, which we state below for the convenience of the reader. Recall from Definition \ref{perversejump} that for any perverse sheaf $\sP\in \Perv(T,\kk)$, the cohomology jump loci of $\sP$ are defined as:
\be\label{cjl}\sV^i(T,\sP): = \{\chi \in \spec \Gamma_T \mid H^i(T, \sP \otimes_\kk L_\chi) \neq 0\},\ee
where $L_{\chi}$ is the rank one local system of $\kk_{\chi}$-vector spaces on $T$ associated to the maximal ideal $\chi$ of $\Gamma_T$, with $\kk_{\chi}=\Gamma_T/\chi$ the residue field of $\chi$. Then we have the following. 
\bt  \label{12} Let $\kk$ be a fixed field.  For any perverse sheaf $\sP\in \Perv(T,\kk)$, the cohomology jump loci of $\sP$ satisfy the following propagation package:
\begin{enumerate}
\item[(i)] {\it Propagation property}:
$$ 
\spec \Gamma_T \supseteq \sV^{0}(T,\sP) \supseteq \sV^{-1}(T,\sP) \supseteq \cdots \supseteq \sV^{-N}(T,\sP).
$$
\item[(ii)] {\it Codimension lower bound}: for any $i\geq 0$, $$ \codim \sV^{-i}(T,\sP) \geq i.$$
\item[(iii)]  If  $V$ is an irreducible component of $\sV^0(T,\sP)$ of codimension $d$, then $V \subset \sV^{-d}(T,\sP).$ In particular, $V$ is also an irreducible component of $\sV^{-d}(T,\sP)$;
\item[(iv)]  If $\sV^0(T,\sP)$ has codimension $d\geq 0$, then $$\sV^0(T,\sP)= \sV^{-1}(T,\sP)= \cdots = \sV^{-d}(T,\sP) \neq \sV^{-d-1}(T,\sP).$$  If moreover $\codim \sV^{-k}(T,\sP)>k$ for all $k>d$, then $\sV^0(T,\sP)$ is pure-dimensional.
\item[(v)] {\it Generic vanishing}: there exists a non-empty Zariski open subset $U \subset \spec \Gamma_T $ such that,  for any maximal ideal $\chi\in U$,  $H^{i}(T, \sP\otimes_{\kk} L_\chi)=0$ for all $i\neq 0$.
\item[(vi)] {\it Signed Euler characteristic property}:  $$  \chi(T,\sP)\geq 0.$$
Moreover, the equality holds if and only if $\sV^0(T,\sP) \neq \spec \Gamma_T$.
\end{enumerate} 
\et

\begin{proof}
For any $\sP\in \Perv(T,\kk)$, Theorem \ref{GL} shows that $\sM_*(\sP)$ is a bounded coherent complex of $\Gamma_T$-modules with cohomology concentrated in degree zero. Then we can define as in Section \ref{alg} the corresponding jumping ideals  $$J^i(T,\sP):= J^i(\sM_*(\sP))$$ and the jump loci \be\label{ajl} \sV^i(T,\sP):= \sV^i( \sM_*(\sP)).\ee  
Then our theorem follows from Proposition \ref{prop}, provided that we show that the two definitions (\ref{cjl}) and (\ref{ajl}) of jump loci $\sV^i(T, \sP)$ coincide. 

Note that $\sL_T \otimes_{\Gamma_T}  a^*\kk_\chi \cong L_\chi$, hence we have that \footnote{Here, and in the sequel, it is important to tensor by a local system in order for the projection formula for the derived pushforward to hold.}
 $$ Ra_* (\sP \otimes_\kk \sL_T) \otimes_{\Gamma_T} \kk_\chi 
 \cong Ra_*(\sP \otimes_\kk \sL_T \otimes_{\Gamma_T} a^* \kk_\chi)\cong Ra_*(\sP \otimes_\kk L\chi).$$
Let us now take degree $i$-cohomology on the first and last terms. From the last term we get $H^i(X, \sP\otimes_\kk L_\chi)$, which is non-zero if and only if $\chi \in \sV^i(T, \sP)$ of definition (\ref{cjl}). On the other hand, it follows from Remark \ref{alt} that $H^i(Ra_* (\sP \otimes_\kk \sL_T) \otimes_{\Gamma_T} \kk_\chi)$ is non-zero if and only if  $\chi \in \sV^i( \sM_*(\sP)))$ of definition (\ref{ajl}).  \end{proof}


\subsection{Applications} \label{apl}
In this subsection, we assume that $\kk=\bC$.
In this case, the following decomposition theorem holds.
\bt \label{decomposition} \cite{BBD,CM} Let $f:X\to Y$ be a proper map of complex algebraic varieties. There exists an isomorphism in $D^b_c(Y, \bC)$:
\be \label{simple}  
Rf_* \IC_X \cong  \bigoplus_j \text{  } ^p \cH^j(Rf_* \IC_X)[-j],
\ee
where $\IC_X$ is the intersection cohomology complex on $X$ with $\bC$-coefficients.
The decomposition is finite, and $j$ ranges in the interval $[-r(f), r(f)]$, where $$r(f)=\dim (X \times_Y X) -\dim X$$ is the defect of the semi-smallness of the proper map $f$.
\et 
\br
It is known that the decomposition theorem may fail for finite field coefficients, see \cite{JMW}. 
\er

Let $X$ be a smooth connected complex $n$-dimensional algebraic variety, and let $f:X \to T$ be a proper algebraic map from $X$ to the complex affine torus $T= (\bC^*)^N$. Assume that $n>r(f)$. 
Then it was shown in \cite{LMW} that the induced homomorphism $f_*:\pi_1(X) \to \pi_1(T)$ is non-trivial, and let us denote its image by $\bZ^m$ (this image is a free abelian group). 

\bd The {\it cohomology jump loci of $X$ with respect to the proper map $f: X\to T$} are defined as:
$$ \sV^{i}(X,f): =  \{\chi \in \spec\Gamma_T \mid   H^i(X, f^* L_\chi) \neq 0\} ,$$ 
 where $L_\chi$ is as before the rank-one $\bC$-local system on $T$ associated to $\chi$.\ed
Note that $\codim \sV^0(X,f)= m$, since $X$ is connected. In fact, it is clear that $H^0(X, f^{-1}L_\chi)\neq 0$ if and only $f^{-1}L_\chi$ is the constant sheaf. 

\bc  \label{complex}
Let $f : X \to T$ be a proper map of complex algebraic varieties, with $X$ smooth of complex dimension $n$.  Then the cohomology jump loci $\sV^*(X,f)$ have the following properties:
\begin{itemize}
\item[(1)] {\it Propagation property}:
$$ 
  \sV^{n-r(f)}(X,f) \supseteq \sV^{n-r(f)-1}(X,f) \supseteq \cdots \supseteq \sV^{0}(X, f) \supseteq \{ 1 \};
$$
\item[(2)] {\it Codimension lower bound}: for any $ 0 \leq i \leq n$, $$ \codim \sV^{n-i}(X,f) \geq i-r(f);$$

\item[(3)] {\it Generic vanishing}: there exists a non-empty Zariski open subset $U \subset \spec \Gamma_T
$ such that, for any maximal ideal $\chi\in U$,  $H^{i}(X, f^{*} L_\chi)=0$ for all $ i < n-r(f)$. 
\end{itemize} 
\ec
 
\begin{proof} Since $f$ is proper, we have $Rf_!=Rf_*$.
By the projection formula \cite[Theorem 2.3.29]{D2} (or simply because $ L_\chi$ is a local system), we get that
\begin{equation*}
\begin{split}
  Rf_*(\bC_X[n] \otimes_\bC f^*L_\chi) &\cong Rf_* (\bC_X[n]) \otimes_\bC L_\chi  \\ 
  &\cong   \bigoplus_{j\in [-r(f),r(f)]} (\text{  } ^p \cH^j(Rf_* \bC_X[n])[-j]) \otimes_\bC L_\chi,
  \end{split}
 \end{equation*}
 where the second isomorphism uses the decomposition theorem.
Therefore, $$\sV^i(X,f)=   \bigcup_{j\in [-r(f), r(f)]} \sV^{i-n -j} (T, \text{} ^p \cH^j(Rf_* \bC_X[n])) .$$
Then the claim follows from Theorem \ref{main}. 
\end{proof}
 
\bc \label{betti} Let $f : X \to (\cc^*)^N$ be a proper map of complex algebraic varieties, with $X$ smooth of complex dimension $n>r(f)$. Then we have that
 \begin{center}
$b_i(X)>0$ for any $i\in [0, n-r(f)]$,
\end{center} 
and \begin{center}
$b_1(X)\geq n-r(f)$.
\end{center}  If, moreover, $r(f)=0$, then 
\begin{center}
$b_i(X)>0$ exactly for $i\in [0,n]$ and $(-1)^n \chi(X)\geq 0$.
\end{center} 
\ec

\begin{proof}
 Let $\bZ^m$ be 
the image of the non-trivial homomorphism $f_*:\pi_1(X) \to \pi_1(T)$.  Then $\codim \sV^0(X,f)= m$. The codimension lower bound implies that 
$$   \codim \sV^0(X,f) \geq n-r(f), $$
 hence $b_1(X)\geq m \geq n-r(f) .$

Next note that $\{1\} \in \sV^0(X,f) $. Therefore, the propagation property yields that $b_i(X)>0$ for any $i\in [0, n-r(f)]$.

When $r(f)=0$, $Rf_* \bC_X[n]$ is a perverse sheaf on $T$, hence $$(-1)^n\chi(X)= \chi(X, \bC_X[n]) = \chi(T,Rf_* \bC_X[n]) \geq 0,$$
where the last inequality follows from Theorem \ref{main}$(vi)$.
\end{proof} 
\br Let $f : X \to (\cc^*)^N$ be a proper map of complex algebraic varieties, with $X$ smooth of complex dimension $n$. Then we have the following lower bounds for $r(f)$: \begin{center}
$ r(f) \geq n-b_1(X)$ and $ r(f) \geq n- \max\{q\geq 0 \mid b_i(X)>0  \text{ for any } 0\leq i \leq q\}.$ 
\end{center} 
\er


\section{Abelian duality spaces}\label{abd}

\subsection{Partially abelian duality space}
Let $X$ be a connected finite CW complex, and denote $\pi_1(X)$ by $G$. Let $\phi: G\to G'$ be a non-trivial homomorphism to an abelian group $G'$. There is a canonical $\bZ [G^\prime]$-local coefficient system $\sL_{\phi}$ on $X$, whose monodromy action is given by the composition of $G\overset{\phi}{ \to} G^\prime$ with the natural multiplication  $G^\prime\times \bZ[ G^\prime] \to \bZ [G^\prime]$.  
\bd\label{dpad} We call $X$ a {\it partially abelian duality space of dimension $n$ with respect to} $\phi$, if the following two conditions are satisfied:
\begin{itemize}
\item[(a)] $H^i(X, \bZ [G'])=0$ for $i\neq n$, 
\item[(b)] $H^n(X, \bZ [G'])$ is a non-zero torsion-free $\bZ$-module. 
\end{itemize}
\ed

\br It will be shown in Remark \ref{redu} that the non-vanishing condition in Definition \ref{dpad}(b) is {\it redundant}.
\er

\br Note that there is a canonical $\bZ[G']$-module isomorphism
$$H^i (X, \bZ[G']) \cong H^i(X,\sL_{\phi}),$$
for any $i$.
\er

\br \label{abelianization} If, in Definition \ref{dpad}, we let $G^\prime=G^{ab}=H_1(X, \bZ)$ and $\phi$ is the abelianization map, then $X$ is called an {\it abelian duality space}, see \cite{DSY}.
\er 

\bd  Let $\kk$ be either $\bZ$ or a field. Set $\Gamma_{G'}= \kk[G^\prime].$  The {\it cohomology jump loci of $X$ with $\kk$-coefficients with respect to $\phi$} are defined as:
 $$\sV^i(X, \phi): = \{\chi \in \spec \Gamma_{G'} \mid H^i(X, \phi^* L_\chi) \neq 0\},$$
  where to any point $\chi \in \spec \Gamma_{G'}$ 
one associates the rank one local system  $\phi^* L_\chi:=L_{\phi^* \chi}$ on $X$. 
 \ed 

\bt \label{pad} Let $X$ be a partially abelian duality space of dimension $n$ with respect to $\phi$. Let $\kk$ be either $\bZ$ or a field. Then the cohomology jump loci of $X$ with $\kk$-coefficients with respect to $\phi$ satisfy the following propagation package:
\begin{enumerate}
\item[(i)] {\it Propagation property}:
$$ 
\spec \Gamma_{G^\prime} \supseteq \sV^{n}(X,\phi) \supseteq \sV^{n-1}(X,\phi) \supseteq \cdots \supseteq \sV^{0}(X,\phi).
$$
\item[(ii)] {\it Codimension lower bound}: for any $i\geq 0$, $$ \codim \sV^{n-i}(X,\phi)= b_1(G')-\dim  \sV^{n-i}(X,\phi) \geq i.$$
\item[(iii)]  If  $V$ is an irreducible component of $\sV^n(X,\phi)$ with codimension  $d$, then $V \subset \sV^{n-d}(X,\phi).$ In particular, $V$ is also an irreducible component of $\sV^{n-d}(X,\phi)$.
\item[(iv)]  If $\sV^n(X,\phi)$ has codimension $d\geq 0$, then $$\sV^n(X,\phi)= \sV^{n-1}(X,\phi)= \cdots = \sV^{n-d}(X,\phi) \neq \sV^{n-d-1}(X,\phi).$$  If moreover $\codim \sV^{n-k}(X,\phi)>k$ for all $k>d$, then $\sV^n(X,\phi)$ is pure-dimensional.

\item[(v)] {\it Generic vanishing}: there exists a non-empty Zariski open subset $U \subset \spec \Gamma_{G^\prime}$ such that,  for any maximal ideal $\chi\in U$,  $H^{i}(X, \phi^* L_\chi)=0$ for all $i\neq n$.
\item[(vi)] {\it Signed Euler characteristic property}:  $$ (-1)^n \chi(X)\geq 0.$$
Moreover, the equality holds if and only if $\sV^n(X,\phi) \neq \spec \Gamma_{G^\prime}$.
\end{enumerate} 
 \et

\begin{proof}

Since  $X$ is a partially abelian duality space of dimension $n$ with respect to $\phi$, we have that  $H^i (X, \bZ[G'])=0$ for $i\neq n$.  
When $\kk$ is a field, using the fact that $H^n (X, \bZ[G'])$ is torsion free and the universal coefficient theorem, 
we get
\be  \label{ucte}
 H^i (X, \kk[G'])=0 \text{ for } i\neq n.
\ee

Let $a: X\to pt$ be the constant map. Let $\sL^{\kk}_\phi:=\sL_{\phi} \otimes_{\bZ} \kk$ be the canonical $\kk[G']$-local system on $X$ induced by $\phi$. Then there is an isomorphism of $\kk[G']$-modules
\be\label{re} H^i (X, \kk[G']) \cong H^i(X,\sL^{\kk}_\phi),\ee for all $i$. So, by (\ref{ucte}), $Ra_*\sL^{\kk}_\phi$ is quasi-isomorphic to  $H^n(X, \kk[G'])[-n]$, i.e., $H^n(X, \kk[G'])$ viewed as a complex concentrated in degree $n$.

For any maximal ideal $\chi \in \spec \kk[G']$ with residue field $\kk_{\chi}$ we have  $\sL_\phi^\kk \otimes_{\kk[G']}  a^*\kk_\chi = \phi^* L_\chi$, hence 
 \be\label{side} (Ra_*  \sL_\phi^\kk) \otimes_{\kk[G']} \kk_\chi 
 \cong Ra_*(\sL_\phi^\kk \otimes_{\kk[G']} a^* \kk_\chi)\cong Ra_* (\phi^* L\chi).\ee
 
Now, taking cohomology on both sides of the above isomorphism, and using Remark \ref{alt}, we have the following.
\begin{lemma}\label{equal}
Under the above notations, for all $i$, 
$$\sV^i(X, \phi)=\sV^i(H^n(X, \kk[G'])[-n])$$
as subsets of $\spec \Gamma_{G'}$.
\end{lemma}

Then Theorem \ref{pad} follows from the above lemma and Proposition \ref{prop}.
\end{proof}

By the definition of cohomology jump loci of modules in Section \ref{alg}, $\sV^i(M)=\emptyset$ when $i>0$. Therefore, we get the following consequence of Lemma \ref{equal}.
\begin{cor}\label{dim}
If $X$ is a partial abelian duality space of dimension $n$ with respect to $\phi$, then $\sV^i(X, \phi)=\emptyset$ for all $i>n$. 
\end{cor}

\br\label{redu} Following the lines of the above proof, one can in fact show that the condition $H^n(X, \bZ [G'])\neq 0$ in Definition \ref{dpad} of a (partially) abelian duality space is redundant (i.e., it is always satisfied). Indeed, assuming that condition (a) of Definition \ref{dpad} holds and fixing a field $\kk$, say $\kk=\bC$, we have (\ref{ucte}). If, by absurd, we assume $H^n(X, \bZ [G'])= 0$, then we get by (\ref{re}) that the complex $Ra_*\sL^{\kk}_\phi$ is quasi-isomorphic to  the zero complex, with $a:X \to pt$ the constant map. Hence by (\ref{side}), the complex $Ra_* (\phi^* L\chi)$ is quasi-isomorphic to the zero complex, for any maximal ideal $\chi \in \spec \kk[G']$. This leads to a contradiction since, for $\chi$ the trivial character (corresponding to the constant sheaf), we have $H^0(Ra_* (\phi^* L\chi))=H^0(X,\bC)\neq 0$.

\er

\subsection{Universal coefficient theorem}\label{uct}
The following result is crucial in the proof of Theorem \ref{abel} and of its generalization from Section \ref{gad}. 
  \bp \label{inverse}
Let $N^\bullet$ be a bounded complex of finitely generated free $\bZ[G']$-modules, where $G'$ is a free abelian group. 
Assume that for $\kk=\bQ$, as well as for $\kk=\Fp$ for every prime number $p$, we have  
 $H^i(N^\bullet\otimes_\bZ \kk)=0$ for all $i\neq n$. Then $H^i(N^\bullet)=0$ for $i\neq n$ and $H^n(N^\bullet)$ is a torsion-free $\bZ$-module. 
\ep

\begin{proof}
First, we recall the fact that every $\bZ$-module $M$ fits into a short exact sequence
$$0\to M_{\it torsion}\to M\to M/M_{\it torsion}\to 0,$$
which splits when $M$ is finitely generated. Here, $M_{\it torsion}$ denotes the torsion $\bZ$-submodule of $M$.

Let $m$ be the largest integer such that $H^m(N^\bullet)\neq 0$. Suppose that $m>n$. Since $\bQ$ is a flat $\bZ$-module, any $\bZ$-module monomorphism $\bZ\to M$ induces a monomorphism   $\bQ\to M\otimes_\bZ \bQ$. Therefore, the assumption that $H^i(N^\bullet\otimes_\bZ \bQ)=0$ for all $i\neq n$ implies that $H^m(N^\bullet)$ is a torsion $\bZ$-module. Thus, there exists a prime number $p$ and a nonzero element $x_1\in H^m(N^\bullet)$ such that $px_1=0$. Since tensor  is a right exact functor, we also have that 
$$H^m(N^\bullet)\otimes_\bZ \Fp \cong H^m(N^\bullet\otimes_\bZ \Fp)=0.$$
Hence, there exists a non-zero element $x_2\in H^m(N^\bullet)$ such that $px_2=x_1$. Similarly,  there exists a non-zero element $x_3\in H^m(N^\bullet)$ such that $px_3=x_2$, and so on. The existence of such a sequence $\{x_1, x_2, \ldots \}$ contradicts the noetherian property of $H^m(N^\bullet)$. In fact, if we let $J_i=\{x\in H^m(N^\bullet) \ | \ p^ix=0\}$, then
$$0=J_0\subset J_1\subset J_2\subset \cdots$$
is a sequence of  $\bZ[G']$-submodules of $H^m(N^\bullet)$. Since $N^\bullet$ is a complex of finitely generated free  $\bZ[G']$-modules, $H^m(N^\bullet)$ is a noetherian $\bZ[G']$-module. This means that the inclusions $J_i\subset J_{i+1}$ will eventually becomes equalities. However, the elements $x_i$ constructed above satisfy $x_i\in J_i\setminus J_{i-1}$, which yields a contradiction. Therefore $m \leq n$, and hence $H^i(N^\bullet)=0$ for $i> n$.

Since $H^i(N^\bullet\otimes_\bZ \bQ)=0$ for all $i < n$, it follows as above that 
$H^i(N^\bullet)$ is a torsion $\bZ$-module when $i<n$. Therefore, in order to complete the proof, it suffices to show that $H^i(N^\bullet)$ is a torsion-free $\bZ$-module for any $i \leq n$. For any prime number $p$, $\bZ\xrightarrow[]{p}\bZ\to F_p$ is a free resolution of $\Fp$. Thus, 
$$N^\bullet\xrightarrow[]{p} N^\bullet\to N^\bullet\otimes_\bZ \Fp\to N^\bullet[1]$$
forms a distinguished triangle in the derived category of $\bZ$-modules. So we have a long exact sequence
$$\cdots\to H^{i-1}(N^\bullet\otimes_\bZ \Fp)\to H^{i}(N^\bullet)\xrightarrow[]{p} H^i(N^\bullet)\to H^i(N^\bullet\otimes_\bZ \Fp)\to\cdots$$
Since $H^i(N^\bullet\otimes_\bZ \Fp)=0$ for $i<n$, it follows that $H^{i}(N^\bullet)\xrightarrow[]{p} H^i(N^\bullet)$ is injective for any $i\leq n$. Hence, if $i\leq n$, the $\bZ$-module $H^i(N^\bullet)$ does not contain any element of order $p$, for any prime number $p$. Therefore,  $H^i(N^\bullet)$ is a torsion-free $\bZ$-module for any $i \leq n$. 
\end{proof}



\subsection{
Detecting abelian duality spaces}\label{gad}
In this section we prove the following theorem, which implies Theorem \ref{abel} by Remark \ref{abelianization}.
\begin{theorem}\label{veryaffine}
Let $X$ be a smooth complex $n$-dimensional quasi-projective variety, and let $f: X\to T=(\bC^*)^N$ be a proper semi-small morphism (e.g., a finite morphism or a closed embedding). Then:
\begin{enumerate}
\item $X$ is a partially abelian duality space of dimension $n$ with respect to $f_*: \pi_1(X)\to \pi_1(T)=\bZ^N$. 
\item Moreover, if the mixed Hodge structure on $H^1(X, \bQ)$ is pure of type $(1,1)$, then $X$ is an abelian duality space. 
\end{enumerate}
\end{theorem}

Before prooving Theorem \ref{veryaffine}, let us introduce some notations. 
 Let $\Gamma_{\bZ}=\bZ[\pi_1(T)]\cong \bZ[t_1^{\pm 1}, \ldots, t_N^{\pm 1}]$. We denote the canonical $\Gamma_{\bZ}$-local system on $T$ by $\sL_\bZ$. The {\it integral Mellin transformation functor} $\sM_*: D_c^b(T, \bZ)\to D^b_{coh}(\Gamma_\bZ)$ is defined as in Section \ref{MT} by
$$\sM_*(\sP):=Ra_*(\sL_{\bZ} \otimes_\bZ \sP),$$
where $a: T\to pt$ is the constant map to a point. (In these notations, we emphasize the coefficients $\bZ$, as the torus $T$ is fixed.)

\begin{proof}[Proof of Theorem \ref{veryaffine}] \item[(1)] Let $\kk$ be the field $\bQ$ or $\Fp$. Since $f$ is proper and semismall, $Rf_*(\kk_X[n])$ is a $\kk$-perverse sheaf on the complex affine torus $T$, e.g., see  \cite[Example 6.0.9]{Schu}. It then follows from Theorem \ref{GL} that
$$H^i(\sM_*(Rf_*(\kk_X[n])))=0 \text{ for } i\neq 0$$
or, equivalently,
$$H^i(\sM_*(Rf_*(\kk_X)))=0 \text{ for } i\neq n.$$
By using the projection formula, one has that
$$\sM_*(Rf_*(\bZ_X))\stackrel{L}{\otimes}_\bZ \kk \cong \sM_*(Rf_*(\bZ_X \otimes_\bZ \kk))\cong \sM_*(Rf_*(\kk_X)),$$
which has nonzero cohomology only in degree $n$. Since $X$ is homotopy equivalent to a finite CW complex, $\sM_*(Rf_*(\bZ_X))$ can be realized as a bounded complex of finitely generated free $\Gamma_\bZ$-modules. Notice that, by definition, $$H^i(\sM_*(Rf_*(\bZ_X)))\cong H^i(X,f^*\sL_{\bZ}) \cong H^i(X, \bZ[G']),$$ where $G'=\pi_1(T)$ and the homomorphism $\pi_1(X)\to G'$ is induced by the map $f:X \to T$. Therefore, Theorem \ref{veryaffine} (1) follows from Proposition \ref{inverse}.

Finally, it follows by Remark \ref{redu} that $H^n(X, \bZ[G']) \neq 0$. Let us, however, also indicate a more geometric proof of this fact. Since $H^n(X, \bZ[G'])$ is a torsion-free $\bZ$-module, it suffices to prove that $H^n(X, \bQ[G']) \neq 0$. As above, we have that
$$H^n(X, \bQ[G']) \cong H^0(\cM_*(Rf_*(\bQ_X[n]))) \cong \cM_*(Rf_*(\bQ_X[n])),$$
where the last equality uses the $t$-exactness of the Mellin transformation and the fact that $Rf_*(\bQ_X[n])$ is a $\bQ$-perverse sheaf on $T$. Finally, the arguments of \cite[Proposition 3.4.6]{GL} can be adapted to our setting to show that $\cM_*(Rf_*(\bQ_X[n])) = 0$ if, and only if, $Rf_*(\bQ_X[n])=0$, which is clearly not the case.

\medskip

\item[(2)]  Let us now assume that the mixed Hodge structure on $H^1(X, \bQ)$ is pure of type $(1,1)$. 
Denote the first Betti number of $X$ by $b_1(X)=l$. 

Let $\alb_X: X\to \Alb(X)$ be the generalized Albanese map (e.g., see \cite[Proposition 4]{Iit}). Since $H^1(X, \bQ)$ is of type $(1,1)$ and since $b_1(X)=l$, the semi-abelian variety $\Alb(X)$ is isomorphic to $(\bC^*)^l$. Notice that the generalized Albanese map of an affine torus is an isomorphism, which we will consider as an identity. By the functoriality of the Albanese morphism, the map $f: X\to T$ induces a commutative diagram
$$
\xymatrix{
X\ar^{f}[r]\ar^{\alb_X}[d]&T\ar^{\alb_T=\id}[d]\\
\Alb(X)=(\bC^*)^l\ar^{\overline{f}}[r]& \Alb(T)=T.
}
$$
Since $(\bC^*)^l$ is a smooth variety, $\overline{f}: (\bC^*)^l\to T$ is separable. Thus, the fact that $f: X\to T$ is proper implies that $\alb_X: X\to (\bC^*)^l$ is proper. By a simple dimension count, one can see that the semi-smallness assumption on $f$ implies that $\alb_X$ is semi-small. Altogether, $\alb_X$ is proper and semi-small. 

By the definition of the generalized Albanese map, $\alb_{X,*}: H_1(X, \bZ)\to H_1((\bC^*)^l, \bZ)$ induces an isomorphism between $H_f:=H_1(X, \bZ)/{\it torsion}$ and $H_1((\bC^*)^l, \bZ)\cong \bZ^l$. 

Set $H_t:={\it Tors}(H_1(X, \bZ))$, where ${\it Tors}$ denotes the torsion subgroup. Let 
$\Gamma_X:=\bZ [H_1(X,\bZ)]$, $\Gamma_f:=\bZ [H_f]$ and $\Gamma_t:=\bZ [H_t]$. As defined in the beginning of this section, there exist a canonical $\Gamma_X$-local system $\sL_X$ and a canonical $\Gamma_f$-local system $\sL_f$ on $X$. 
Any non-canonical splitting $H_1(X,\bZ) \cong H_f \oplus H_t$ 
induces a non-canonical isomorphism 
$$\sL_X\cong \sL_{f}\otimes_\bZ \sL_t,$$
where $\sL_t$ is the $\Gamma_t$-local system on $X$ defined via the epimorphism $\pi_1(X) \ra H_1(X,\bZ) \ra H_t$ (induced by the splitting).
Denote the integral Mellin transformation functor on $(\bC^*)^l$ by $\sM_*$. Then, by using the projection formula for $\alb_{X}$, we have a (non-canonical) isomorphism
$$Ra_*(\sL_X)\cong \sM_*(R( \alb_{X})_*(\sL_t))$$
in $D^b_{coh}(\Gamma_X)$, with $a:X \to pt$ the constant map to a point. Here, we regard $\sL_t$ as a local system of finitely generated free $\bZ$-modules on $X$. 
Next note that we have a canonical isomorphism:
$$
\sM_*(R(\alb_X)_*(\sL_t))\stackrel{L}{\otimes}_\bZ \kk \cong \sM_*(R(\alb_X)_*(\sL_t\stackrel{L}{\otimes}_\bZ \kk)).
$$
We can now apply the same argument as in the proof of part (1), with $f$ and $\bZ_X$ replaced by $\alb_X$ and $\sL_t$, respectively, to conclude that $H^i(X, \sL_X)=0$ when $i\neq n$ and $H^n(X, \sL_X)$ is torsion-free. This completes the proof of the second part of Theorem \ref{veryaffine}. 
\end{proof}


\section{Examples}\label{sex}

\bex \label{affine}  Let $X$ be  an $n$-dimensional smooth closed subvariety of the complex affine torus $T$ (that is, $X$ is a {\it very affine manifold}).  The closed embedding map $f:X\to T$ is clearly a proper semi-small map, and hence $X$ is a partially Abelian duality space with respect to the induced homomorphism $f_*: \pi_1(X) \to \pi_1(T)$.  So the properties listed in Theorem \ref{pad}  hold. 
If, moreover, the mixed Hodge structure on $H^1(X, \bQ)$ is pure of type $(1,1)$, then $X$ is an abelian duality space, and hence $X$ also satisfies the properties of Theorem \ref{ad}. 
\eex

\bex \label{hyperplane} Let $X$ be  the complement of a union of $(N-n)$ irreducible hypersurfaces  in $(\bC^*)^n$, e.g., the complement of an essential hyperplane arrangement, or of a toric hyperplane arrangement.  It is easy to see that one can always find a closed embedding  $f:X \to (\bC^*)^N$  such that $X$ is an algebraic closed submanifold of $(\bC^*)^N$ and the induced map on the first $\bZ$-homology  groups is an isomorphism. In particular,  the mixed Hodge structure on $H^1(X, \bQ)$ is pure of type $(1,1)$, hence $X$ is a abelian duality space.  Then the properties listed in Theorem \ref{ad} hold.
\eex

\bex \label{hypersurface} 
Consider a hypersurface $V$ in $\CP$, where $V=V_{0}\cup \cdots \cup V_{N}$ has $N+1$ irreducible components $V_{i}=\lbrace f_{i}=0 \rbrace$, $i=0,\cdots,N$. Here $f_{i}$ is a reduced homogeneous polynomial of degree $d_{i}$.  Assume, moreover, that the hypersurface $V$ is {\it essential}, i.e., $V_{0}\cap \cdots \cap V_{N} = \emptyset.$ It is clear that if $V$ is essential, then $n\leq N$.

Set $\gcd(d_0,\cdots,d_N)=d$.
Consider the well-defined map 
 $$f=\big( f_{0}^{d / d_0}, f_{1}^{d / d_1}, \cdots, f_{N}^{d/d_N}\big): \CP  \to \mathbb{CP}^{N}.$$
 The divisor $\sum_{i=0}^N d/d_i \cdot V_i$ defines an ample line bundle on $\CP$. If follows from \cite[Corollary 1.2.15]{La} that $f$ is a finite map.  Taking the restriction of $f$ over $X=\CP\setminus V$, we obtain a finite map $f: X \ra (\bC^{\ast})^{N}= \mathbb{CP}^N \setminus \bigcup_{i=0}^N D_i$, hence $f$ is proper and semi-small. 
Since  the mixed Hodge structure on $H^1(X, \bQ)$ is pure of type $(1,1)$, we get by Theorem \ref{veryaffine} that $X$ is an abelian duality space.  Then the properties listed in Theorem \ref{ad}  hold.    The signed Euler characteristic property $(-1)^n\chi(X)\geq 0$ in this case was first proved in \cite[Corollary 5.12]{LM}.
\eex

\bex \label{ample}  Let $Y$ be a smooth complex projective variety, and let $\sL$ be a very ample line bundle on $Y$. Consider a $N$-dimensional sub-linear system $\vert E\vert $ of $\vert \sL \vert$ such that $E$ is base point free over $Y$. Then a basis $\{s_0, s_1, \cdots, s_N\}$ of $E$ gives  a well-defined morphism  $$\varphi_{\vert E \vert}:   Y \to \mathbb{CP}^N .$$ Each  $\{ s_i=0 \} $ defines a  hypersurface $V_i$ in $Y$. In particular, $\bigcap_{i=0}^N V_i = \emptyset$.  Since $\sL $ is very ample, $\varphi_{\vert E \vert}$ is a finite morphism. In fact, 
if $\varphi_{\vert E \vert}$ is not finite then there is a subvariety $Z \subset Y$ of positive dimension which is contracted by $\varphi_{\vert E \vert}$ to a point. Since $\sL= \varphi_{\vert E \vert}^* \sO_{\mathbb{CP}^N}(1)$, we see that $\sL$ restricts to a trivial line bundle on $Z$. In particular, $\sL_{\vert  Z }$ is not ample, and by \cite[Proposition 1.2.13]{La}, neither is $\sL$, which contradicts with our assumption that $\sL$ is very ample. (Here we use the proof from \cite[Corollary 1.2.15]{La}.)

 Taking the restriction $f$ of $\varphi_{\vert E \vert}$ over  $X= Y\setminus \bigcup_{i=0}^N V_i $, we get a map
\begin{align*}
 f: X & \lra T=(\bC^*)^N \\
x  &\mapsto (\dfrac{s_1}{s_0}, \dfrac{s_2}{s_0}\cdots,\dfrac{s_N}{s_0}),
\end{align*}
which is finite, hence  proper and semi-small. 
Then $X$ is a partially abelian duality space for the induced homomorphism $f_*: \pi_1(X) \to \pi_1(T)$.  So, the properties listed in Theorem \ref{pad} hold. 

If, moreover, $H^1(Y,\bQ)=0$, then the mixed Hodge structure on $H^1(X, \bQ)$ is pure of type $(1,1)$, hence $X$ is an abelian duality space. So, in this case the properties listed in Theorem \ref{ad} hold.
\eex

\bex  In all examples considered above, the map $f:X \to T$ is finite, hence proper and semi-small.  Hence $Rf_{\ast}=Rf_{!}$ is $t$-exact, i.e., $Rf_{\ast}=Rf_{!}$ preserves perverse sheaves (e.g., see \cite[Corollary 5.2.15]{D2}). Let $X$ be any of the spaces considered in the above examples. 
 Let $W$ be a connected closed subvariety of $X$ of pure dimension $k$,  which is locally a complete intersection (e.g., $W$ is a hypersurface in $X$). Then for any local system  $L$ of finite rank on  $W$, $L[k]$ is a perverse sheaf on $W$ (e.g., see \cite[Theorem 5.1.20]{D2}). By taking $L$ to be the constant sheaf $\bC_W$, the same arguments as in Corollary \ref{betti} yield the following properties: 
 \begin{enumerate}
\item[(i)] $b_i(W) > 0$ for any $0 \leq i \leq k$,
\item[(ii)] $b_1(W) \geq  k$,
\item[(iii)]  $(-1)^{k}\chi(W) \geq 0$. 
\end{enumerate}
 More generally, if $W \overset{j}{\hookrightarrow} X$ is a closed subvariety of $X$ of pure dimension $k$, one can consider the intersection cohomology complex $IC_W$ on $W$, which is a perverse sheaf, and apply Theorem \ref{12} to the perverse sheaf $Rf_*(j_*IC_W)$ on $T$.  
For example, Theorem \ref{12}$(i)$ shows that the intersection cohomology group $\mathrm{ IH} ^i(W,\bQ)\neq 0$ for any $0\leq i\leq k$. 
Moreover, Theorem \ref{12}$(vi)$ yields that $$\sum_{i=0}^k (-1)^{i+k }\dim \mathrm{ IH}^i (W,\bQ) \geq 0.$$ 
\eex

\section{Compact K\"ahler manifolds as abelian duality spaces}\label{projective}

In this section, we show that compact complex tori are the only compact K\"ahler manifolds that are abelian duality spaces. 

\begin{theorem}\label{av}
Let $X$ be a compact K\"ahler manifold. Then $X$ is an abelian duality space if and only if $X$ is a compact complex torus. In particular, abelian varieties are the only complex projective manifolds that are abelian duality spaces.
\end{theorem}
\begin{proof}
Obviously, a compact complex torus is an abelian duality space. Conversely, let us first assume  that $X$ is a complex projective manifold, and an abelian duality space. Denote the complex dimension of $X$ by $n$. 
Then $X$ is a closed oriented smooth manifold of real dimension $2n$. 

We consider $\bC$-coefficient cohomology jump loci $\sV^i(X)\subseteq \spec \Gamma_{G^{ab}}$, where $\Gamma_{G^{ab}}=\bC[G^{ab}]$. Since $\sV^0(X)=\{\mathbf{1}\}$, consisting of the trivial representation, we obtain by Poincar\'e duality that $\sV^{2n}(X)=\{\mathbf{1}\}$. Since $\sV^{i}(X)=0$ for all $i>2n$, by Corollary \ref{dim}, $X$ is an abelian duality space of dimension at least $2n$. On the other hand, since $\sV^{2n}(X)=\{\mathbf{1}\}$, by Theorem \ref{pad} (i), it follows that $X$ can not be an abelian duality space of dimension less than $2n$. Hence $X$ is an abelian duality space of dimension exactly $2n$. 

Again by Theorem \ref{pad} (i) and $\sV^{2n}(X)=\{\mathbf{1}\}$, we have the equalities $\sV^0(X)=\sV^1(X)=\cdots=\sV^{2n}(X)=\{\mathbf{1}\}$. Therefore, by \cite[Theorem 2.1]{Wa17}, the Albanese map $\alb_X: X\to \Alb(X)$ is surjective and $R^i\alb_{X*}(\bC_X)$ are local systems on $\Alb(X)$ for all $i\in \bN$. Thus, $\dim X\geq \dim \Alb(X)$, or equivalently $n\geq \frac{1}{2}b_1(X)$. On the other hand, by Theorem \ref{Bettinumbers}, we have $b_1(X)\geq 2n$. Therefore, $b_1(X)=2n$ and the Albanese map $\alb_X$ is surjective and generically finite. Since $R^i\alb_{X*}(\bC_X)$ are all local systems, we then get that $R^i\alb_{X*}(\bC_X)=0$ for $i\geq 1$. Therefore, all the fibers of $\alb_X$ are zero-dimensional. In other words, $\alb_X$ is quasi-finite. A proper quasi-finite map is finite. Since $R^0\alb_{X*}(\bC_X)$ is a local system on $\Alb(X)$, the finite map $\alb_X$ has no ramification locus, and hence is a covering map. Any finite cover of an abelian variety is also an abelian variety, thus proving the assertion in the smooth projective case. 


Let us now discuss the proof in the K\"ahler case.
The only step in the above proof that requires $X$ being projective is when we quote \cite[Theorem 2.1]{Wa17}. We next show how this particular theorem can be generalized to compact K\"ahler manifolds by using the arguments in the proof of \cite[Proposition 4.1]{Wa16}. Since the support of $R^0\alb_{X*}(\bC_X)$ is equal to the image of $\alb_X$, if $R^0\alb_{X*}(\bC_X)$ is a local system on $\Alb(X)$, then $\alb_X$ is surjective. Therefore, in order to complete the proof of the theorem it suffices to show the following: 
\begin{lemma}\label{62}
Let $X$ be a compact K\"ahler manifold. Suppose $\sV^i(X)$ are zero-dimensional for all $i$. Then $R^i\alb_{X*}(\bC_X)$ are local systems on $\Alb(X)$, for all $i$. 
\end{lemma}
\noindent{\it Proof of Lemma \ref{62}.} \ 
By the decomposition theorem of proper maps of K\"ahler manifolds (\cite{Sai}), there exists a decomposition
\begin{equation}\label{decomposition}
R\alb_{X*}(\bC_X)\cong \bigoplus_{1\leq j\leq l}M_j[d_j]
\end{equation}
where each $M_j$ is a perverse sheaf which underlies an irreducible pure polarizable Hodge module on $\Alb(X)$. By definition, 
$$\bigcup_{i,j}\sV^i(\Alb(X), M_j)=\bigcup_{i}\sV^i(X).$$
Thus, $\sV^i(\Alb(X), M_j)$ is zero-dimensional for all $i$ and $j$. Let us fix $j$. We claim that the support of $M_j$ cannot be contained in any proper subtorus of $\Alb(X)$. Indeed, if the support of $M_j$ is contained in a proper subtorus $S$ of $\Alb(X)$, then every $\sV^i(\Alb(X), M_j)$ is invariant under the translations by elements in the image of $\Char(\Alb(X)/S)\to \Char(\Alb(X))$, induced by the natural projection $\Alb(X)\to \Alb(X)/S$. This contradicts the fact that all $\sV^i(\Alb(X), M_j)$ are  zero-dimensional. The same argument also shows that the support of $M_j$ can not be contained in any translate of a proper subtorus of $\Alb(X)$. 

In order to complete the proof of Lemma \ref{62}, we make use of the following statement derived from the proof of \cite[Theorem 4.1]{Wa16}:
\begin{prop}\label{63}
Let $\sT$ be a compact complex torus. Let $M$ be a perverse sheaf on $\sT$ that underlies an irreducible polarizable Hodge module of geometric origin. Assume that the support of $M$ is not contained in any translate of a proper subtorus of $\sT$. Then there exists a finite cover $\pi: \sT'\to \sT$, a holomorphic map $\psi: \sT'\to \sT'_{\textrm{alg}}$ to an abelian variety (the algebraic reduction of $\sT'$) and a constructible complex $N\in D^b_c(\sT'_{\textrm{alg}}, \bC)$, such that $\pi^*(M)\simeq \psi^*(N)$ in $D^b_c(\sT', \bC)$. 
\end{prop}


Applying Proposition \ref{63} for $\sT=\Alb(X)$ and $M=M_j$, we have a finite cover $\pi: \sT'\to\Alb(X)$, a holomorphic map $\psi: \sT'\to \sT'_{\textrm{alg}}$, and $N\in D^b_c(\sT'_{\textrm{alg}}, \bC)$, such that $\pi^*(M_j)\cong \psi^*(N)$. Since $\sV^i(\Alb(X), M_j)$ are zero-dimensional for all $i$, we have that $\sV^i(\sT', \pi^*(M_j))$ are also zero-dimensional (see the proof of \cite[Lemma 3.2]{Wa16}). By the projection formula, it follows that all $\sV^i(\sT'_{\textrm{alg}}, N)$ are zero-dimensional (see the proof of \cite[Proposition 4.1]{Wa16}). Now, by the proof of \cite[Theorem 2.1]{Wa17}, the constructible complex $N$ is isomorphic to a shift of a local system on $\sT'_{\textrm{alg}}$, and hence both $\pi^*(M_j)$ and $M_j$ are shifts of local systems. Therefore, we have shown that each direct summand $M_j$ in the decomposition (\ref{decomposition}) is a shift of local system. This completes the proof of the lemma. 
\end{proof}

One can further ask the following questions inspired by the above theorem. 
\begin{que}
Does there exists a closed orientable manifold that is an abelian duality space, but not a real torus?
\end{que}

\begin{que}\label{23}
If $X$ is a 
duality space, what can be said about its topology? 
For example, is 
the Euler characteristic of $X$ signed? 
\end{que}

If $X$ is aspherical and a closed oriented manifold, then it is a duality space. Moreover, the following holds:
\bp\label{24} Let $X$ be a closed oriented manifold of dimension $m$, then $X$ is a duality space if and only if $X$ is aspherical.
\ep
\begin{proof} If $X$ is a duality space of dimension $k$, then, by definition, the compactly supported cohomology $H_c^*(\widetilde{X},\bZ)$ of the universal cover $\widetilde{X}$ of $X$ is concentrated in degree $k$. On the other hand, since $\widetilde{X}$ is an $m$-dimensional manifold, Poincar\'e duality yields that the homology $H_*(\widetilde{X},\bZ)$ is concentrated in degree $m-k$. Since $H_0(\widetilde{X},\bZ) \neq 0$,  this further implies $k=m$.
Therefore, $\widetilde{X}$ is a simply-connected manifold with trivial reduced homology, so it follows by Hurewicz that ${X}$ is aspherical.
\end{proof}


\br Motivation for the second part of Question \ref{23} comes from the {\it Hopf and Singer conjectures}, e.g., see \cite[Ch.11]{Lu} for an overview. Hopf conjectured that if $M$ is a closed manifold of real dimension $2n$, 
with negative sectional curvature, then $(-1)^n \chi(M) \geq 0$.
Jost-Zuo \cite{JZ} proved Hopf's conjecture in the K\"ahler context. More precisely, 
they showed that if $M$ is 
a compact K\"ahler manifold of complex dimension $n$ and non-positive 
sectional curvature, then $(-1)^n \chi(M) \geq 0$. 
Hopf's conjecture was strenghtened by Singer, who claimed that if $M^{2n}$ 
is a closed aspherical manifold, then $(-1)^n \chi(M) \geq 0$. 
Note that if $M$ carries a Riemannian metric with non-positive sectional 
curvature, then $M$ is aspherical by Hadamard's Theorem. 
So Proposition \ref{24} gives a reformulation of Singer's conjecture from the point of view of homological duality.
\er


Back to the realization problem mentioned in the Introduction, one would ultimately like to classify all quasi-projective (or quasi-K\"ahler) manifolds that are (abelian) duality spaces.



\begin{thebibliography}{ADMSP}

\bibitem[BBD82]{BBD} 
A. A. Beilinson, J. Bernstein, P. Deligne, {\it Faisceaux pervers}, Ast\'{e}risque 100, Paris, Soc.
Math. Fr. 1982. 

\bibitem[BSS17]{BSS} B. Bhatt, S. Schnell, P. Scholze, {\it 
Vanishing theorems for perverse sheaves on abelian varieties, revisited}, arXiv:1702.06395.

\bibitem[BG09]{BG} W. Bruns, J. Gubeladze, {\it 
Polytopes, rings, and K-theory.}  Springer Monographs in Mathematics. Springer, Dordrecht, 2009.

 
  
\bibitem[BW15]{BW} N. Budur, B. Wang,
 {\it Cohomology jump loci of differential graded Lie algebras}, Compositio Math. 151, no. 8 (2015), 1499-1528.
 
\bibitem[BW17]{BW17} N. Budur, B. Wang, {\it Absolute sets and the Decomposition Theorem}, arXiv:1702.06267.
 
\bibitem[BE73]{BE} R. Bieri, B. Eckmann, {\it Groups with homological duality generalizing Poincar\'e duality}, Invent. Math. 20 (1973), 103-124.

\bibitem[CM09]{CM} M. A. A. de Cataldo, L. Migliorini, {\it The decomposition theorem, perverse sheaves and the topology of algebraic maps}. Bull. Amer. Math. Soc. (N.S.) 46 (2009), no. 4, 535-633.

\bibitem[CH01]{CH} J. Chen, C. Hacon, {\it Characterization of abelian varieties},  
Invent. Math. {143} (2001), no. 2, 435--447.




\bibitem[DS17]{DS} G. Denham, A. Suciu, {\it
Local systems on arrangements of smooth, complex algebraic hypersurfaces}, arXiv:1706.00956. 

\bibitem[DSY15]{DSY} G. Denham, A. Suciu, S. Yuzvinsky, {\it Abelian duality and propagation of resonance}, arXiv:1512.07702.


\bibitem[Di04]{D2} A. Dimca, {\it Sheaves in Topology},
Universitext, Springer-Verlag, Berlin, 2004.


\bibitem[E95]{E} D. Eisenbud, {\it Commutative algebra. With a view toward algebraic geometry}. Graduate Texts in Mathematics, 150. Springer-Verlag, New York, 1995.


\bibitem[FK00]{FK} J. Franecki, M. Kapranov, {\it  The Gauss map and a noncompact Riemann-Roch formula for constructible sheaves on semiabelian varieties},  Duke Math. J. 104 (2000), no. 1, 171-180. 

\bibitem[GL96]{GL} O. Gabber, F. Loeser, {\it Faisceaux pervers $\ell$-adiques sur un tore}, Duke Math. J. 83 (1996), no. 3, 501-606.

\bibitem[GL87]{GLa}  M. Green, R. Lazarsfeld, {\it Deformation theory, generic vanishing theorems, and some conjectures of Enriques, Catanese and Beauville} Invent. Math. {90} (1987), no. 2, 389--407.

\bibitem[GL91]{GrLa}  M. Green, R. Lazarsfeld,  {\it Higher obstructions to deforming cohomology groups of line bundles}, J. Amer. Math. Soc. {4} (1991), no. 1, 87--103.



\bibitem[Ha77]{Ha} R. Hartshorne, {\it Algebraic geometry}, Graduate Texts in Mathematics, no. 52, Springer-Verlag, New York-Heidelberg, 1977.


\bibitem[Iit76]{Iit} S. Iitaka, {\it  Logarithmic forms of algebraic varieties}. J. Fac. Sci. Univ. Tokyo Sect. IA Math. 23 (1976), no. 3, 525-544. 

\bibitem[JZ00]{JZ} J. Jost, K. Zuo,  \emph{Vanishing theorems for $L^2$-cohomology on infinite coverings of compact K\"ahler manifolds and applications in algebraic geometry}, Comm. Anal. Geom. {8} (2000), no. 1, 1--30.

\bibitem[JMW14]{JMW}
D. Juteau, C. Mautner and G. Williamson, {\it Parity sheaves}, J. Amer. Math. Soc. 27 (2014), no. 4, 1169-1212.

 \bibitem[Kra14]{Kra} T. Kr\"{a}mer, {\it Perverse sheaves on semiabelian varieties}, Rend. Semin. Mat. Univ. Padova 132 (2014), 83-102.

\bibitem[KW15]{KW}
T. Kr\"{a}mer, R. Weissauer, {\it Vanishing theorems for constructible sheaves on abelian varieties}, J. Algebraic Geom. 24 (2015), no. 3, 531-568.

\bibitem[La04]{La}
R. Lazarsfeld,  {\it Positivity in algebraic geometry. I. Classical setting: line bundles and linear series}, Ergebnisse der Mathematik und ihrer Grenzgebiete. 3. Folge. A Series of Modern Surveys in Mathematics [Results in Mathematics and Related Areas. 3rd Series. A Series of Modern Surveys in Mathematics], 48. Springer-Verlag, Berlin, 2004.  

\bibitem[LM14]{LM} Y. Liu, L. Maxim, {\it Characteristic varieties of hypersurface complements}. Adv. Math. 306 (2017), 451-493.


\bibitem[LMW17]{LMW} Y. Liu, L. Maxim, B. Wang, {\it Generic vanishing for semi-abelian varieties and integral Alexander modules}, 	arXiv:1707.09806.

\bibitem[Lu02]{Lu} W. L\"uck, {\it $L^2$-invariants: Theory and Applications to Geometry and K-Theory}, Ergebnisse der Mathematik und ihrer Grenzgebiete. 3. Folge. A Series of Modern Surveys in Mathematics, 44. Springer-Verlag, Berlin, 2002.

\bibitem[Ma89]{Ma} H. Matsumura, {\it Commutative ring theory.} Translated from the Japanese by M. Reid. Second edition. Cambridge Studies in Advanced Mathematics, 8. Cambridge University Press, Cambridge, 1989. xiv+320 pp. ISBN: 0-521-36764-6.



 

\bibitem[Sai90]{Sai} M. Saito, {\it Decomposition theorem for proper K\"ahler morphisms}, Tohoku Math. J. (2) 42 (1990), no. 2, 127-147. 

\bibitem[Sch15]{Sch}
C. Schnell, {\it  Holonomic $\mathcal{D}$-modules on abelian varieties},
Publ. Math. Inst. Hautes \'{E}tudes Sci. 121 (2015), 1-55.  

\bibitem[Schu03]{Schu} J. Sch\"{u}rmann,  {\it Topology of Singular Spaces and Constructible Sheaves}, Monografie Matematyczne 63, Birkh\"auser Verlag, Basel 2003.


\bibitem[Wa16]{Wa16}
B. Wang, {\it Torsion points on the cohomology jump loci of compact K\"ahler manifolds}, Math. Res. Lett. 23 (2016), no. 2, 545-563.

\bibitem[Wa17]{Wa17}
B. Wang, {\it Algebraic surfaces with zero-dimensional cohomology support locus}, arXiv:1702.05169,  Taiwanese J. Math. (to appear).

\bibitem[We16]{We}
R. Weissauer, {\it Vanishing theorems for constructible sheaves on abelian varieties over finite fields}, Math. Ann. 365 (2016), no. 1-2, 559-578.

\bibitem[Yau]{Yau}
S. T. Yau, {\it On Calabi's conjecture  and some new  results  in  algebraic  geometry}, Nat. Acad. Sci.  U.S.A., 74, 1977, 1798-1799.

\end{thebibliography}
\end{document}